\documentclass{amsart}
\usepackage{amssymb,amsthm,url}
\newtheorem{theorem}{Theorem}[section]
\newtheorem{lemma}[theorem]{Lemma}

\newtheorem{definition}[theorem]{Definition}
\newtheorem{notation}[theorem]{Notation}

\newcommand{\lcm}{\mathop{\mathrm{lcm}}}
\newcommand{\Out}{\mathop{\textrm{Out}}}
\newcommand{\GL}{\mathop{\textrm{GL}}}

\newcommand{\PSL}{\mathop{\textrm{PSL}}}
\newcommand{\PSU}{\mathop{\textrm{PSU}}}
\newcommand{\Sp}{\mathop{\mathrm{Sp}}}
\newcommand{\PSp}{\mathop{\mathrm{PSp}}}
\def\cent#1#2{{\bf C}_{{#1}}({{#2}})}
\newcommand{\Aut}{\mathop{\mathrm{Aut}}}

\begin{document}

\title[Maximal element order of symplectic groups]{The maximum order of the elements  of a finite symplectic group of even characteristic}
\author[P. Spiga]{Pablo Spiga}
\address{Pablo Spiga, Dipartimento di Matematica Pura e Applicata,\newline
 University of Milano-Bicocca, Via Cozzi 53, 20126 Milano, Italy} 
\email{pablo.spiga@unimib.it}


\begin{abstract}
We give an exact formula, as a function of $m$ and $q$, for the maximum order of the elements of the finite symplectic group $\Sp_{2m}(q)$, with $q$ even, and of its automorphism group.
\end{abstract}
\maketitle 

\section{Introduction}

The maximum order of the elements of a finite simple group of Lie type of odd characteristic was computed by Kantor and Seress in~\cite{KS}. Their motivation is computational: some algorithms for computations with a matrix group $G$ require the characteristic of $G$ as an input. There are polynomial time algorithms for computing the characteristic of $G$, but these are often not practical, see~\cite{KS2} or the introduction in~\cite{KS}. So, Kantor and Seress provide in~\cite{KS} an alternative polynomial type algorithm for computing the characteristic of $G$. This algorithm relies on~\cite[Theorem~$1.2$]{KS}, which states that, for a simple group of Lie type $G$ of odd characteristic $p$,  the three largest element orders of $G$ determine uniquely $p$. For more details and for an algorithmic implementation of these results we refer to~\cite{KS}.

The hypothesis of $p$ being odd is essential only for simple classical groups. In fact, for these groups some delicate computations on the order of semisimple elements in maximal tori heavily depend upon this requirement. Section~$2$ in~\cite{KS} describes in details the obstacles for pinning down an exact formula for the maximum order of the elements of a simple classical group of even characteristic. 

In~\cite{DGPS}, as part of a rather different investigation, the authors have determined exact formulae, in any characteristic, for the maximal order of an element in almost simple groups with socle $\PSL_n(q)$ (see Corollary~$2.7$ and Theorem~$2.16$) and $\PSU_n(q)$ (see Lemma~$2.15$ and Theorem~$2.16$). 

In this paper we study the finite symplectic groups $\Sp_{2m}(q)$, with $q$ even. Observe that, for $q$ even, $\Sp_{2m}(q)=\PSp_{2m}(q)$ is simple for $(m,q)\notin\{ (1,2),(2,2)\}$.

\begin{theorem}\label{thrm}
Let $m\geq 1$ and let $q$ be a power of $2$. The maximum order of the  elements of $\Sp_{2m}(q)$ is $M_m(q)$, where $M_m(q)$ is given in Definition~\ref{def}. Moreover, the maximum order of the elements of $\Aut(\Sp_{2m}(q))$ is $6$ if $(m,q)=(1,4)$, $10$ if $(m,q)=(2,2)$, $20$ if $(m,q)=(2,4)$, and $M_m(q)$ otherwise.
\end{theorem}

The definition of $M_m(q)$ is rather cumbersome compared to the odd characteristic case (see~\cite[Table~A.3]{KS}), however we show in Lemma~\ref{estimate} that $q^m<M_m(q)\leq (q^{m+1}-1)/(q-1)$, which might be useful for some  practical purposes.

\begin{definition}\label{def}{\rm Let $m\geq 1$ and let $q$ be a power of $2$. The value of the function $M_m$ on $q$ depends on the parity of $m$ and on whether $q=2$ or $q>2$.  We start with the case that $q>2$. When $m$ is odd, let  $m=2^{i_1}+2^{i_2}+\cdots +2^{i_\ell}$ be the $2$-adic expansion of $m$ and define $$M_{m}(q)=\prod_{j=1}^{\ell}(q^{2^{i_j}}+1).\qquad (m\textrm{ odd and }q>2)$$
Observe that when $m=1+2+\cdots +2^{\ell-1}$ the product $M_m(q)$ is $(q^{2^\ell}-1)/(q-1)$, see~\eqref{eq0} below.

When $m$ is even and $m\geq 4$, let $\ell$ be the largest positive integer with $2^\ell+2^{\ell-1}\leq m$ and define
\[
M_{m}(q)=
\left\{
\begin{array}{lcl}
q^2+1&&\textrm{if }m=2 \textrm{ and }q>2,\\ 
(q^{m-2^\ell+1}-1)(q^{2^{\ell}}-1)/(q-1)&&\textrm{if }m\geq 4\textrm{ is even and }q>2.
\end{array}
\right.
\]

We now turn to the definition of $M_{m}(q)$, for $q=2$. Let $\ell$ be the largest positive integer with $2^\ell-1\leq m$ and write $m_0=m-(2^\ell-1)$. Now, define
\begin{equation}\label{eqdef}
M_{m}(q)=\left\{
\begin{array}{lcl}
q^{m_0}(q^{2^\ell}-1)&&\textrm{if }m_0\leq 3,\\
(q^{2^{\ell-1}+m_0}-1)(q^{2^{\ell-1}}-1)&&\textrm{if }3<m_0<2^{\ell-1}, m_0 \textrm{ odd},\\
q(q^{2^{\ell-1}+m_0-1}-1)(q^{2^{\ell-1}}-1)&&\textrm{if }3<m_0<2^{\ell-1}, m_0 \textrm{ even},\\
(q^{2^{\ell}}+1)(q^{2^{\ell-1}}-1)&&\textrm{if }3<m_0, m_0=2^{\ell-1},\\
(q^{m_0}-1)(q^{2^{\ell}}-1)&&\textrm{if }m_0>\max(3,2^{\ell-1}), m_0 \textrm{ odd},\\
q(q^{m_0-1}-1)(q^{2^{\ell}}-1)&&\textrm{if }m_0>\max(3,2^{\ell-1}), m_0 \textrm{ even}.\\
\end{array}
\right.
\end{equation}

In order to get acquainted with this definition we have tabulated $M_{m}(q)$ in Table~\ref{table}, for $m\leq 20$.
}
\end{definition}

\begin{table}[!!h]
\begin{tabular}{|c|cc|}\hline
   $m$     &$q>2$&$q=2$\\\hline
$1$&$(q^2-1)/(q-1)$&$q^2-1$\\
$2$&$q^2+1$&$q(q^2-1)$\\
$3$&$(q^4-1)/(q-1)$&$q^4-1$\\
$4$&$(q^3-1)(q^2-1)/(q-1)$&$q(q^4-1)$\\
$5$&$(q+1)(q^4+1)$&$q^2(q^4-1)$\\
$6$&$(q^3-1)(q^4-1)/(q-1)$&$q^3(q^4-1)$\\
$7$&$(q^8-1)/(q-1)$&$q^8-1$\\
$8$&$(q^5-1)(q^4-1)/(q-1)$&$q(q^8-1)$\\
$9$&$(q+1)(q^8+1)$&$q^2(q^8-1)$\\
$10$&$(q^7-1)(q^4-1)/(q-1)$&$q^3(q^8-1)$\\
$11$&$(q+1)(q^2+1)(q^8+1)$&$(q^8+1)(q^4-1)$\\
$12$&$(q^5-1)(q^8-1)/(q-1)$&$(q^5-1)(q^8-1)$\\
$13$&$(q+1)(q^4+1)(q^8+1)$&$q(q^5-1)(q^8-1)$\\
$14$&$(q^7-1)(q^8-1)/(q-1)$&$(q^7-1)(q^8-1)$\\
$15$&$(q^{16}-1)/(q-1)$&$q^{16}-1$\\
$16$&$(q^9-1)(q^8-1)/(q-1)$&$q(q^{16}-1)$\\
$17$&$(q+1)(q^{16}+1)$&$q^2(q^{16}-1)$\\
$18$&$(q^{11}-1)(q^8-1)/(q-1)$&$q^3(q^{16}-1)$\\
$19$&$(q+1)(q^2+1)(q^{16}+1)$&$q(q^{11}-1)(q^{8}-1)$\\
$20$&$(q^{13}-1)(q^8-1)/(q-1)$&$(q^{13}-1)(q^8-1)$\\\hline
\end{tabular}
\caption{$M_{m}(q)$, for $m\leq 20$}\label{table}
\end{table}
For the rest of this paper we let $m$ denote a positive integer and we let $q\geq 2$ denote a power of $2$. Moreover, for a finite group $G$ and $g\in G$, we let $|g|$ denote the order of $g$.

\subsection{Structure of the paper}\label{structure} The proof of Theorem~\ref{thrm} is based on a number-theoretic theorem on partitions (Theorem~\ref{thrm2}). In Section~\ref{sec2}, we state Theorem~\ref{thrm2} and, for not breaking the flow of the argument, we prove 
Theorem~\ref{thrm}.  We postpone the proof of Theorem~\ref{thrm2} until Section~\ref{sec4}.

\section{Theorem~\ref{thrm2} and the proof of Theorem~\ref{thrm}}\label{sec2}

Before stating Theorem~\ref{thrm2} we need to introduce some notation.
\begin{notation}\label{not}{\rm For us, a partition  of $m$ of length $\ell$ is an  $\ell$-tuple $(d_1,\ldots,d_\ell)$ of positive integers with $m=d_1+\cdots+d_\ell$. (We consider the empty tuple to be a partition of $0$.)  Moreover, a signed partition of $m$ is a symbol $\wp=(d_1^{\varepsilon_1},\ldots,d_\ell^{\varepsilon_\ell})$ with $(d_1,\ldots,d_\ell)$  a partition of $m$ of length $\ell$ and $\varepsilon_i\in \{-1,1\}$, for each $i\in \{1,\ldots,\ell\}$. We refer to $d_i^{\varepsilon_i}$ as a \textrm{part} of $\wp$ and to $\varepsilon_i$ as the \textrm{sign} of $d_i$. (Our definition of signed partition is unrelated to the definition introduced by Andrews~\cite[page~$567$]{Andrews}. We find this name quite descriptive of the fact that each part of the partition is equipped with a sign, and hence we feel at liberty to borrow this term.)

Let ${m'}\in \{0,\ldots,m\}$ and let $\wp=(d_1^{\varepsilon_1},\ldots,d_\ell^{\varepsilon_\ell})$ be a signed partition of $m-{m'}$. We write
$$L_{m,q}({m'},\wp)=2^{\lceil \log_2(2{m'})\rceil}\lcm(q^{d_1}-\varepsilon_1,q^{d_2}-\varepsilon_2,\ldots,q^{d_\ell}-\varepsilon_\ell).$$
(Here,  ``$\lcm$'' stands for the least common multiple, and we set $2^{\lceil \log_2(2m')\rceil}=1$ when $m'=0$.) We find it helpful to think of $L_{m,q}$ as a function of ${m'}$ and $\wp$.}
\end{notation}

The key ingredient in our proof of Theorem~\ref{thrm} is the following.  

\begin{theorem}\label{thrm2}$M_m(q)=\max(L_{m,q}({m'},\wp)\mid {m'},\wp)$.
\end{theorem}

We postpone the proof of Theorem~\ref{thrm2} to Section~\ref{sec4}. Here we just give a rough estimate on the order of magnitude of $M_m(q)$.  
Given $\ell\geq 1$, we have 
\begin{equation}\label{eq0}
\prod_{i=0}^{\ell-1}(q^{2^i}+1)=\frac{q^{2^{\ell}}-1}{q-1}.
\end{equation}
(Expanding the product on the left hand side we get $q^{2^{\ell-1}}+q^{2^{\ell-1}-1}+\cdots +q+1$, which equals the right hand side.) 

\begin{lemma}\label{estimate}We have $q^m< M_{m}(q)\leq (q^{m+1}-1)/(q-1)$. Moreover, if $m\neq 2$ or if $q=2$, then $M_m(q)>q^{m+2}/(q^2-1)$.
\end{lemma}
\begin{proof}
We start by proving the first part of the statement.
The lower bound follows by comparing $q^m$ with $M_m(q)$ as given in  Definition~\ref{def}. For instance, if $q>2$ and $m$ is odd, then $M_m(q)$  (viewed as a polynomial in $q$) has degree $m$ and has positive coefficients. Thus $M_m(q)>q^m$. The other cases are similar and we omit the details.

The upper bound follows by comparing $(q^{m+1}-1)/(q-1)=q^{m}+q^{m-1}+\cdots +q+1$ with $M_m(q)$. The only computation that is not straightforward is when $q>2$ and $m$ is odd: we discuss this case here in detail (we use the notation established in Definition~\ref{def}). As $i_1,i_2,\ldots ,i_\ell$ are pair-wise distinct, we see that $$M_m(q)=\prod_{j=1}^\ell(q^{2^{i_j}}+1)=q^m+a_{m-1}q^{m-1}+a_{m-2}q^{m-2}+\cdots +a_2q^2+a_1q+1,$$
with $a_1,\ldots,a_{m-1}\in \{0,1\}$. From this we have $M_m(q)\leq \sum_{i=0}^m q^i= (q^{m+1}-1)/(q-1)$. The other cases are similar.

Finally, the second part of the statement follows again with a case-by-case analysis comparing $M_m(q)$ with $q^{m+2}/(q^2-1)$. Each case requires only routine computations comparing the polynomials $(q^2-1)M_m(q)$ and $q^{m+2}$. Here we only prove it when $m$ is odd and $q>2$, which we regard it as the hardest case. So, let $2^{i_1}+\cdots +2^{i_\ell}$ be the $2$-adic expansion of $m$ with $i_1<\cdots <i_\ell$. If $m=1$, then with an easy computation we get $M_1(q)=q+1>q^3/(q^2-1)$. Assume that $m>1$. Observe that $i_1=0$ (because $m$ is odd) and $1=2^{i_1}<2^{i_\ell}\leq m-1$. Now
\begin{eqnarray*}
M_m(q)&=&
q^m\prod_{j=1}^\ell\left(1+\frac{1}{q^{i_j}}\right)\geq 
q^m\left(1+\frac{1}{q^{2^{i_1}}}\right)\left(1+\frac{1}{q^{2^{i_\ell}}}\right)\\
&\geq& q^m\left(1+\frac{1}{q}\right)\left(1+\frac{1}{q^{m-1}}\right)>\frac{q^{m+2}}{q^2-1},
\end{eqnarray*}
where the last inequality follows with a direct computation.
\end{proof}
 The upper bound in Lemma~\ref{estimate} is sharp when $m+1$  is a power of $2$. Now we establish some notation for the proof of Theorem~\ref{thrm} (we follow~\cite[Section~$2$]{DGPS}).

\begin{notation}\label{not1}{\rm 
Let $V=\mathbb{F}_q^{2m}$ be the $2m$-dimensional natural module of $\Sp_{2m}(q)$ over the field $\mathbb{F}_q$ with $q$ elements. So, $V$ is equipped with a non-degenerate symplectic form preserved by $\Sp_{2m}(q)$. 

Let $s$ be a semisimple element of $\Sp_{2m}(q)$. The action of the matrix 
$s$ on $V$ defines the structure of an $\mathbb{F}_{q}\langle s\rangle$-module on $V$.  
Since $s$ is semisimple, $V$ decomposes, by Maschke's theorem, as a direct 
sum of irreducible $\mathbb{F}_{q}\langle s\rangle$-modules.

Now, we make use of a theorem of 
Bertram Huppert~\cite[Satz 2]{H}, which we apply to the semisimple element $s$. By Huppert's Theorem, $V$ 
admits an orthogonal decomposition of the following form: 
\begin{eqnarray}\label{eq:decomp}
V&=&V'\perp ((V_{1,1}\oplus V_{1,1}')\perp \cdots \perp (V_{1,m_1}\oplus V_{1,m_1}'))\perp \cdots \\
&&\perp ((V_{r,1}\oplus V_{r,1}')\perp \cdots \perp (V_{r,m_r}\oplus V_{r,m_r}'))\nonumber\\
&&\perp (V_{r+1,1}\perp \cdots \perp V_{r+1,m_{r+1}})
\perp\cdots \perp (V_{t,1}\perp \cdots \perp V_{t,m_{t}}) \nonumber
\end{eqnarray}
where $V'$ is the  eigenspace of $s$ for the eigenvalue
$1$, of dimension $2m'$ (note that either $V'=0$ or $V'$ is non-degenerate, and hence $V'$ has even dimension),
and each  $V_{i,j}$ is a non-trivial irreducible $\mathbb{F}_{q}\langle s\rangle$-submodule. For every $i\in \{1,\ldots,t\}$, we have $\dim_{\mathbb{F}_q}V_{i,j}=\dim_{\mathbb{F}_q}V_{i,j'}$, for each $j,j'\in \{1,\ldots,m_i\}$.
Moreover, for $i\in \{r+1,\ldots,t\}$, $V_{i,j}$ is non-degenerate of dimension $2d_i$  and $s$ 
induces an element of order dividing $q^{d_i}+1$ on $V_{i,j}$. 
For $i\in\{1,\ldots,r\}$, $V_{i,j}$ and $V_{i,j}'$ are totally isotropic  of dimension $d_i$, 
$V_{i,j}\oplus V_{i,j}'$ is non-degenerate, and $s$ 
induces an element $y_{i,j}$ of order dividing $q^{d_i}-1$ on $V_{i,j}$ while
inducing the adjoint representation $(y_{i,j}^{-1})^{tr}$ on  $V_{i,j}'$ (where 
$x^{tr}$ denotes the transpose of the matrix $x$). For our claims about the 
orders and
for some facts on the structure of the maximal tori of $\Sp_{2m}(q)$  we refer to~\cite{BC,KS} or~\cite[Section~$2$]{DGPS}.  

Note that the orthogonal decomposition in~\eqref{eq:decomp} determines the signed partition  
\[ 
\wp(s)=(\underbrace{d_1^{1},\ldots,d_1^{1}}_{m_1\textrm{ times}},\ldots,\underbrace{d_r^{1},\ldots,d_r^{1}}_{m_r\textrm{ times}},
\underbrace{d_{r+1}^{-1},\ldots,d_{r+1}^{-1}}_{m_{r+1}\textrm{ times}},\ldots,\underbrace{d_t^{-1},\ldots,d_t^{-1}}_{m_t\textrm{ times}}) 
\]
of $m-m'$. 

Finally, from~\cite[Proposition~$2.6$]{DGPS}, we see that if $u\in \GL_{2m}(q)$  is  unipotent and centralizes $s$ then
$$|u|\leq \max( 
2^{\lceil \log_2(2{m'})\rceil},2^{\lceil \log_2({m_1})\rceil},\ldots,2^{\lceil \log_2({m_t})\rceil}
).$$
}
\end{notation}

Given two positive integers $a$ and $b$, we write $(a,b)$ for the greatest common divisor of $a$ and $b$. Moreover, we denote by $(a)_2$ the largest power of $2$ dividing $a$. The following lemma is rather elementary but very useful for what follows.

\begin{lemma}\label{basic}Let $a$ and $b$ be positive integers. Then
\begin{description}
\item[(i)]$(q^a-1,q^b-1)=q^{(a,b)}-1$;
\item[(ii)]$(q^a+1,q^b-1)=\left\{
\begin{array}{lcl}
1&&\textrm{if }(a)_2\geq (b)_2,\\
q^{(a,b)}+1&&\textrm{if }(a)_2<(b)_2;
\end{array}
\right.$
\item[(iii)]$
(q^a+1,q^b+1)=\left\{
\begin{array}{lcl}
1&&\textrm{if }(a)_2\neq  (b)_2,\\
q^{(a,b)}+1&&\textrm{if }(a)_2=(b)_2.
\end{array}
\right.
$
\end{description}
\end{lemma}
\begin{proof}
Part~(i) follows by induction on $\max(a,b)$. In fact, if $|a-b|=0$, then there is nothing to prove. If $|a-b|>0$, then, interchanging the roles of $a$ and $b$ if necessary, we may assume that $a>b$. Now $q^a-1=(q^{a-b}-1)q^b+(q^b-1)$ and hence $(q^a-1,q^b-1)=(q^{a-b}-1,q^b-1)=q^{(a-b,b)}-1=q^{(a,b)}-1$.

Observe that if $x$ is even, then $(x+1,x-1)=1$ and hence $(x^2-1,y)=(x-1,y)(x+1,y)$ for every $y$. Thus, by applying~$(i)$ twice, we get  
\begin{eqnarray*}
q^{(2a,b)}-1&=&(q^{2a}-1,q^{b}-1)=(q^a-1,q^{b}-1)(q^a+1,q^b-1)\\ 
&=&(q^{(a,b)}-1)(q^{a}+1,q^b-1).
\end{eqnarray*}
Now to deduce~(ii) note that $(2a,b)=(a,b)$ if $(a)_2\geq (b)_2$, and $(2a,b)=2(a,b)$ if $(a)_2<(b)_2$.

Finally, part~(iii) follows applying~(ii) to $(q^{a}+1,q^{2b}-1)$ and arguing as above. 
\end{proof}

We are now ready to prove Theorem~\ref{thrm} (except for Theorem~\ref{thrm2}, the proof adapts and follows closely the ideas developed in~\cite[Section~$2$]{DGPS}). 
\begin{proof}[Proof of Theorem~\ref{thrm}]
Let $M$ be the maximum order of the elements of $\Sp_{2m}(q)$. We start by showing that $M_m(q)\leq M$. From the description of the semisimple elements given in Notation~\ref{not1} we see that $\Sp_{2m}(q)$ contains an element $g$ with $|g|=M_m(q)$. For example, assume that $q>2$ and that $m$ is odd, and let $m=2^{i_1}+2^{i_2}+\cdots +2^{i_\ell}$ be the $2$-adic expansion of $m$. From Lemma~\ref{basic}~(iii), we see that $q^{2^{i_1}}+1,\ldots,q^{2^{i_\ell}}+1$ are pair-wise coprime. Then, for $g$, it suffices to take a semisimple element of order $(q^{2^{i_1}}+1)\cdots (q^{2^{i_\ell}}+1)$ in the maximal torus isomorphic to $C_{q^{i_1}+1}\times \cdots \times C_{q^{i_\ell}+1}$ (the direct product of cyclic groups of orders $q^{2^{i_1}}+1,\ldots,q^{2^{i_\ell}}+1$). We give another example: assume that $q=2$ and let $2^\ell$ be the largest power of $2$ with $2^\ell-1\leq m$. Write  $m=m_0+(2^\ell-1)$ and suppose that $m_0\leq 3$. Then, for $g$, it suffices to take $g=su=us$, with $s$  a semisimple element  of $\Sp_{2(m-m_0)}(q)$ of order $(q+1)\cdots (q^{2^{\ell-1}}+1)$ in the maximal torus isomorphic to $C_{q+1}\times \cdots \times C_{q^{2^{\ell-1}}+1}$, and with $u$ a unipotent element of $\Sp_{2m_0}(q)$ of order $1$ if $m_0=0$, $2$ if $m_0=1$, $4$ if $m_0=2$,  or $8$ if $m_0=3$ (the existence of $u$ follows by a direct inspection in $\Sp_2(2)$, $\Sp_4(2)$ and $\Sp_6(2)$). The other cases are similar and we leave them to the reader. 

Next, we show that $M\leq M_m(q)$. Let $g$ be an element of $\Sp_{2m}(q)$ with $|g|=M$ and write $g=su=us$, with $s$ semisimple and $u$ unipotent. We use Notation~\ref{not1} for $s$ and $u$. In particular, we have
$$|s|\leq \lcm(q^{d_i}-\varepsilon_i\mid i\in \{1,\ldots,t\})$$
and
\begin{equation}\label{equ}
|u|\leq \max(2^{\lceil \log_2(2m')\rceil},2^{\lceil \log_2(m_1)\rceil},\ldots, 2^{\lceil \log_2(m_t)\rceil}).
\end{equation}

We show that, by replacing $g$ if necessary, we may assume that $m_i=1$ for each $i\in \{1,\ldots,t\}$. We do this in two separate claims.

\smallskip

\noindent\textsc{Claim~$1$. }Replacing $g$ with an element $g'$ having $|g'|=|g|$, we may assume that $m_i=1$ for each $i\in \{1,\ldots,r\}$.

\smallskip

\noindent A computation shows that, for every $a,b\geq 1$, $q^{a}-1$ divides $q^{ab}-1$ and $(q^{ab}-1)/(q^a-1)\geq 2^{\lceil\log_2(b)\rceil}$ (see~\cite[Lemma~$2.4$~(i)]{DGPS}).

Suppose that for some $i\in\{1,\ldots,r\}$ we have $m_i>1$. Then replacing the action of 
$g$ on $(V_{i,1}\oplus V_{i,1}')\oplus \cdots\oplus (V_{i,m_i}\oplus V_{i,m_i}')$ with 
the action given by a semisimple element of order $q^{d_im_i}-1$ (and so having only 
two totally isotropic irreducible $\mathbb{F}_q\langle s\rangle$-submodules), 
we obtain an element  $g'$ such that $|g|$ divides $|g'|$ and  $m_i=1$. In particular, 
replacing $g$ by $g'$ if necessary, we may assume that $g=g'$.~$_\blacksquare$

\smallskip 

\noindent\textsc{Claim~$2$. }Replacing $g$ with an element $g'$ having $|g'|=|g|$, we may assume that $m_i=1$ for each $i\in \{r+1,\ldots,t\}$.

\smallskip 

\noindent  A computation shows that, for $a\geq 1$ and for $b$ odd, $q^{a}+1$ divides $q^{ab}+1$ and $(q^{ab}+1)/(q^a+1)\geq 2^{\lceil\log_2(b)\rceil}$ for $(q,a,b)\neq (2,1,3)$. Moreover, 
for $a\geq 1$ and for $b$ even, $q^{a}+1$ divides $q^{ab}-1$ and $(q^{ab}-1)/(q^a+1)\geq 2^{\lceil\log_2(b)\rceil}$ for $(q,a,b)\neq (2,1,2)$. (See~\cite[Lemma~$2.4$~(ii) and ~(iii)]{DGPS}.)

Suppose that for some $i\in\{r+1,\ldots,t\}$ we have $m_i>1$. With an argument similar to the proof of Claim~$1$, we may assume that $q=2$ and $(d_i,m_i)\in \{(1,3),(1,2)\}$.

Suppose that $(d_{i},m_{i})=(1,3)$. 
The element $g$ induces on $W=V_{i,1}\perp V_{i,2}\perp V_{i,3}$ an element of order 
dividing $(q+1)2^{\lceil \log_2(3)\rceil}=3\cdot 2^2$. 
Let $g'$ be the element acting as $g$ on $W^\perp$, inducing an element of order $3$ 
on $V_{i,1}$ and inducing a regular unipotent element on $V_{i,2}\perp V_{i,3}$. 
Now, $g'$ induces on $W$ an element of order $(q+1)2^{\lceil\log_2(4)\rceil}=3\cdot 2^2$. 
Therefore $|g|=|g'|$ and so, we may replace $g$ by $g'$ (note that in doing so the dimension 
of $V'$ increases by $2$ and $m_{i}$ decreases from $3$ to $1$). 

Suppose that $(d_{i},m_{i})=(1,2)$. 
The element $g$ induces on $W=V_{i,1}\perp V_{i,2}$ an element of order dividing 
$(q+1)2^{\lceil \log_2(2)\rceil}=6$. Let $g'$ be the element acting as $g$ on 
$W^\perp$, inducing an element of order $3$ on $V_{i,1}$ and inducing an element 
of order $2$ on $V_{i,2}$. Now, $g'$ induces on $W$ an element of order $6$. 
Therefore $|g|=|g'|$ and so, we may replace $g$ by $g'$.~$_\blacksquare$

\smallskip

From Claims~$1$ and~$2$, we have $m_i=1$ for every $i\in \{1,\ldots,t\}$, and hence $|u|\leq 2^{\lceil\log_2(2m')\rceil}$ by~\eqref{equ}. 
Thus, from Theorem~\ref{thrm2}, we obtain $$M=|g|=|s||u|\leq L_{m,q}(m',\wp(s))\leq M_m(q).$$
This concludes the proof of the first statement.

We are then left with computing the maximum order of the elements of $\Aut(\Sp_{2m}(q))$. Write $q=2^f$ and let $M$ be the maximum order of the elements of $\Aut(\Sp_{2m}(q))$. From~\cite[Table~$5$, page~xvi]{ATLAS},  we have $\Aut(\Sp_{2m}(q))\cong (\Sp_{2m}(q)\rtimes \langle \phi\rangle).\Gamma$, where $\phi$ is a generator of the group of field automorphisms and $\Gamma$ is the group of automorphisms of the Dynkin diagram. Hence $|\Gamma|=2$ if $m=2$, and $|\Gamma|=1$ if $m\neq 2$. 

Let $g\in \Aut(\Sp_{2m}(q))$ with $|g|=M$. If $g\in \Sp_{2m}(q)$, then from the first part of the theorem we get $M=M_m(q)$. Thus, it suffices to study the case that $g\notin\Sp_{2m}(q)$. Suppose that $g=\varphi x$ with $x\in \Sp_{2m}(q)$ and with $\varphi$ a non-identity field automorphism of order $e>1$. In particular, $e$ is a divisor of $f$. 

Let $\mathbb{F}$ be the algebraic closure of the field $\mathbb{F}_{q}$. By Lang's theorem, there exists $a\in \Sp_{2m}(\mathbb{F})$ with $a^{\varphi}a^{-1}=x$. Observe that
\begin{eqnarray*}
(a^{-1}g^ea)^\varphi&=&a^{-\varphi}(x^{\varphi^{e-1}}\cdots x^{\varphi} x)^\varphi a^{\varphi}=a^{-\varphi}(x^{\varphi^{e}}\cdots x^{\varphi^2} x^\varphi)a^{\varphi}\\
&=&(a^{-1}x^{-1})(xx^{\varphi^{e-1}}\cdots x^{\varphi^2} x^\varphi)(xa)
=a^{-1}(x^{\varphi^{e-1}}\cdots x^\varphi x)a=a^{-1}g^ea.
\end{eqnarray*}
Thus $a^{-1}g^ea$ is invariant under the field automorphism $\varphi$. Hence $a^{-1}g^ea\in \Sp_{2m}(q^{1/e})$ and, from the first part of the theorem applied to $\Sp_{2m}(q^{1/e})$, we get $|a^{-1}g^ea|\leq M_m(q^{1/e})$. Since $a^{-1}g^ea$ is conjugate to $g^e$, we have
\begin{equation}\label{idiot}
|g|=e|g^e|=e|a^{-1}g^ea|\leq eM_m(q^{1/e})\leq e\frac{q^{m/e+1}-1}{q^{1/e}-1},
\end{equation}
where the last inequality follows from Lemma~\ref{estimate}. Now it is a computation to verify that, for $m\neq 1$ and $(f,e,m)\notin\{(2,2,2),(3,3,2)\}$, we have $q^m\geq e(q^{m/e+1}-1)/(q-1)$ and hence $M=|g|\leq M_{m}(q)$ by the lower bound in  Lemma~\ref{estimate}. For $(f,e,m)\in \{(2,2,2),(3,3,2)\}$, a computation with the computer algebra system \texttt{magma}~\cite{magma} shows that the maximal element order of $\Sp_4(4)\rtimes \langle \phi\rangle$ is $17=M_2(4)$ and of $\Sp_4(8)\rtimes\langle \phi\rangle$ is $65=M_2(8)$. For $m=1$, from~\eqref{idiot}, we get $|g|\leq eM_1(q^{1/e})=e(q^{1/e}+1)$. Now, another computation shows that $e(q^{1/e}+1)\leq q+1=M_1(q)$ except for $(f,e)=(2,2)$. Clearly, the maximal element order of $\Aut(\Sp_2(4))$ is $6$, which is one of the exceptions in the statement of this theorem. 

It remains to consider the case $g\in \Aut(\Sp_{2m}(q))\setminus (\Sp_{2m}(q)\rtimes \langle\phi\rangle)$. In particular, $m=2$. Observe that $g^2\in \Sp_4(q)\rtimes \langle \phi\rangle$ and that $M_2(q)=q^2+1$ if $q>2$ and $M_2(2)=6$ if $q=2$. Now, we subdivide the proof into two subcases depending on whether $g^2\in \Sp_4(q)$ or $g^2\notin \Sp_4(q)$. Suppose that $g^2\notin \Sp_4(q)$. Then $g^2=\varphi x$ for some $x\in \Sp_4(q)$ and some field automorphism $\varphi$ of order $e> 1$. The same argument as in the previous two paragraphs shows that $|g|=2|g^2|\leq 2eM_2(q^{1/e})$, which is bounded above by $q^2+1$ for $q>4$. For $q=4$, with \texttt{magma} we see that the maximal element order of $\Aut(\Sp_4(4))$ is $20$, which is one of the exceptions in the statement of this theorem.

Finally, suppose that $g^2\in \Sp_{4}(q)$. Since $g\notin\Sp_4(q)$, the element $g$ projects to an element of order $2$ of $\Out(\Sp_4(q))$. Now, $\Out(\Sp_4(q))$ is cyclic of order $2f$ generated by the  ``extraordinary graph'' automorphism. In particular, if $f$ is even, then $g^2\notin \Sp_4(q)$. Hence $f$ is odd. Assume that $g^2$ has odd order. Then $g$ is centralized by the outer automorphism $g^{|g|/2}$ of order $2$. In particular, $g^2\in \cent{\Sp_4(q)}{g^{|g|/2}}\cong {^2}B_2(q)$ (the last isomorphism follows from~\cite[Proposition~$4.9.1$]{GLS}). Now, from~\cite{Suzuki}, we see that $|g^2|\leq q+\sqrt{2q}+1$. So, $M=|g|\leq 2(q+\sqrt{2q}+1)\leq M_2(q)$, for $q>2$. The maximal element order of $\Aut(\Sp_4(2))$ is $10$, which is one of the exceptions in the statement of the theorem. To conclude, suppose that $g^2$ has even order. A detailed analysis of the elements of even order of $\Sp_4(q)$ shows that $|g^2|\leq 2(q+1)$. Now $|g|\leq 4(q+1)\leq M_2(q)$ for $q>4$. As we have already considered $\Sp_4(2)$ and $\Sp_4(4)$, the proof is complete. 
\end{proof}

\section{Proof of Theorem~\ref{thrm2}}\label{sec4}

Before proceeding with the main result of this section (namely, Theorem~\ref{thrm2}) we single out a rather technical lemma that will be used in its proof.
\begin{lemma}\label{babylonians}
Let $d_1,\ldots,d_\ell$ be positive integers and suppose that $(d_1)_2,\ldots,(d_\ell)_2$ are pair-wise distinct. Let $2^{x_1}+\cdots+2^{x_t}$ be the $2$-adic expansion of $d_1+\cdots +d_\ell$. Then $\prod_{i=1}^\ell(q^{d_i}+1)\leq\prod_{j=1}^t(q^{2^{x_j}}+1)$.
\end{lemma}
\begin{proof}
Relabelling the index sets $\{1,\ldots,t\}$ and $\{1,\ldots,\ell\}$, we may assume that $x_1<\cdots <x_t$ and that $(d_1)_2<\cdots <(d_\ell)_2$. Observe that this yields $(d_1)_2=2^{x_1}$.

We first deal with the case that each of $d_2,\ldots,d_\ell$ is a power of $2$: the general statement will then easily follow by induction. Note that this case includes (vacuously) the case $\ell=1$. 

We argue by induction on $d_1$. Suppose that  $d_1$ is itself a power of $2$. Then the summands of $d_1+\cdots+d_\ell$ already give its $2$-adic expansion. Thus $t=\ell$, $\{d_1,\ldots,d_\ell\}=\{2^{x_1},\ldots,2^{x_\ell}\}$ and there is nothing to prove. 

Suppose that $d_1$ is not a power of $2$, that is, $2^{x_1}=(d_1)_2<d_1$. Let $2^x$ be the largest power of $2$ with $2^x\leq d_1$. Clearly, $x_1<x$. Write $d_1'=d_1-2^x$ and note that $(d_1')_2=(d_1)_2$. Assume that $2^x\neq d_k$, for every $k\in \{2,\ldots,\ell\}$. Then  $(d_1')_2,(2^x)_2,(d_2)_2,\ldots,(d_\ell)_2$ are pair-wise distinct and $d_1'<d_1$. Since the $2$-adic expansion of $d_1'+2^x+d_2+\cdots+d_\ell$ is still $2^{x_1}+\cdots +2^{x_t}$, we conclude  by induction that $(q^{d_1'}+1)(q^{2^x}+1)\prod_{i=2}^\ell(q^{d_i}+1)\leq \prod_{j=1}^t(q^{2^{x_j}}+1)$. As $q^{d_1}+1<(q^{d_1'}+1)(q^{2^x}+1)$, we have $\prod_{i=1}^\ell(q^{d_i}+1)\leq\prod_{j=1}^t(q^{2^{x_j}}+1)$. 

Next, assume that $2^x=d_k$, for some $k\in \{2,\ldots,\ell\}$.
Let $s$ be the largest non-negative integer with $d_{k+j+1}=2d_{k+j}$, for every $j\in \{0,\ldots,s-1\}$. (For instance, $s=0$ exactly when $d_{k+1}>2d_k$, and $s=1$ exactly when $d_{k+1}=2d_k$ and $d_{k+2}>2d_{k+1}$.) Recalling that $d_2,\ldots,d_\ell$ are powers of $2$ and that $2^x=d_k$, we have $2^x+d_k+d_{k+1}+\cdots+d_{k+s}=2^{s+1}d_k$ and we obtain that 
$$d_2+d_3+\cdots +d_{k-2}+d_{k-1}+\, 2^{s+1}d_k\,+d_{k+s+1}+d_{k+s+2}+\cdots+d_{\ell-1} +d_\ell$$
is the $2$-adic expansion of $2^x+d_2+d_3+\cdots +d_\ell$. 
From this it follows that the elements
$(d_1')_2,(d_2)_2,\ldots,(d_{k-1})_2,(2^{s+1}d_k)_2,(d_{k+1})_2,\ldots,(d_\ell)_2$ are pair-wise distinct. As $d_1'<d_1$, by induction, we have
\begin{equation}\label{hard2}
(q^{d_1'}+1)
\left(
\prod_{i=2}^{k-1}
(q^{d_i}+1)
\right)
(q^{2^{s+1}d_k}+1)
\left(\prod_{i=k+1}^\ell(q^{d_i}+1)\right)\leq \prod_{j=1}^t(q^{2^{x_j}}+1).
\end{equation}

We now show that 
\begin{equation}\label{hard1}
(q^{d_1}+1)(q^{d_k}+1)(q^{d_{k+1}}+1)\cdots (q^{d_{k+s}}+1)\leq (q^{d_1'}+1)(q^{2^{s+1}d_k}+1).
\end{equation}
Let $A$ be the left hand side of~\eqref{hard1}. Applying~\eqref{eq0} (with $q$ replaced by $q^{d_k}$) we get
\begin{eqnarray*}
A&=&(q^{d_1}+1)\prod_{i=0}^s\left((q^{d_k})^{2^i}+1\right)=(q^{d_1}+1)\frac{q^{2^{s+1}d_k}-1}{q^{d_k}-1}<(q^{d_1}+1)\frac{q^{2^{s+1}d_k}+1}{q^{d_k}-1}.
\end{eqnarray*}
Now, recalling that $d_1=d_1'+2^x=d_1'+d_k$ and $d_1'<2^x$, it is elementary to check that 
$$q^{d_1}+1\leq (q^{d_1'}+1)(q^{d_k}-1),$$
from which~\eqref{hard1} immediately follows. 

We now return to the proof of the lemma. From~\eqref{hard1}, we get
$$\prod_{i=1}^\ell(q^{d_i}+1)\leq (q^{d_1'}+1)\left(\prod_{i=2}^{k-1}(q^{d_i}+1)\right)(q^{2^{s+1}d_k}+1)\left(\prod_{i=k+1}^\ell(q^{d_i}+1)\right)$$ 
and hence $\prod_{i=1}^\ell(q^{d_i}+1)\leq \prod_{j=1}^t(q^{2^{x_j}}+1)$ by~\eqref{hard2}. This concludes the case when each of $d_2,\ldots,d_\ell$ is a power of $2$.

Next, we argue by induction on $\ell$. Recall that $\ell\geq 2$. Let $2^{y_1}+\cdots +2^{y_r}$ be the $2$-adic expansion of $d_2+\cdots+d_\ell$. By induction, we have 
\begin{equation}\label{hard3}
\prod_{i=2}^\ell(q^{d_i}+1)\leq \prod_{j=1}^r(q^{2^{y_j}}+1).
\end{equation}
Now, the integers $d_1,2^{y_1},\ldots,2^{y_r}$ add up to $d_1+\cdots+d_\ell$, the $2$-powers $(d_1)_2,2^{y_1},\ldots,2^{y_r}$ are pair-wise distinct  and $d_1$ is the only number which is not (necessarily) a power of $2$. Thus, by the case that we discussed above, we have
\begin{equation}\label{hard4}
(q^{d_1}+1)\prod_{j=1}^r(q^{2^{y_j}}+1)\leq \prod_{j=1}^t(q^{x_j}+1).
\end{equation}
Now the induction follows from~\eqref{hard3} and~\eqref{hard4}. 
\end{proof}

For simplifying  some of the arguments in the proof of Theorem~\ref{thrm2} it is convenient to deal  separately with $q=2$ and with small values of $m$. 
\begin{lemma}\label{comp}
If $q=2$ and $m\leq 8$, then $M_m(q)=\max(L_{m,q}(m',\wp)\mid m',\wp)$.
\end{lemma}
\begin{proof}
This follows with a computation with \texttt{magma}~\cite{magma}.
\end{proof}

\begin{proof}[Proof of Theorem~\ref{thrm2}]
For convenience, we let $M$ denote $\max(L_{m,q}(m',\wp)\mid m',\wp)$. It is easy to show, with a case-by-case analysis, that $M\geq M_m(q)$. Here we give full details when $m\geq 4$ is even and $q>2$, the other cases are similar. Let $\ell$ be the largest positive integer with $2^\ell+2^{\ell-1}\leq m$. Take ${m'}=0$ and the signed partition $\wp=(1^{-1},2^{-1},\ldots,(2^{\ell-1})^{-1},(m-2^{\ell}+1)^{1})$ of $m$. Now, a direct application of  Lemma~\ref{basic} and~\eqref{eq0} gives that 
\begin{eqnarray*}
L_{m,q}({m'},\wp)&=&2^{\lceil \log_2(2{m'})\rceil}\lcm(q+1,q^2+1,\ldots,q^{2^{\ell-1}}+1,q^{m-2^{\ell}+1}-1)\\
&=&(q^{m-2^\ell+1}-1)\prod_{i=0}^{\ell-1}(q^{2^i}+1)=(q^{m-2^{\ell}+1}-1)\frac{q^{2^\ell}-1}{q-1}=M_m(q)
\end{eqnarray*}
and hence 
\begin{equation}\label{eq00}
M\geq M_m(q).
\end{equation}

Now we show  that $M\leq M_m(q)$ arguing by induction on $m$. Choose ${m'}$ and $\wp=(d_1^{\varepsilon_1},\ldots,d_\ell^{\varepsilon_\ell})$ with $M=L_{m,q}({m'},\wp)$. We subdivide the proof into eleven claims, from which the result will immediately follow.

\smallskip

\noindent\textsc{Claim~$1$. }Replacing ${m'}$ and $\wp$ if necessary, we may assume that $d_1^{\varepsilon_1},\ldots,d_\ell^{\varepsilon_\ell}$ are pair-wise distinct.

\smallskip

\noindent  We argue by contradiction and we assume that $d_i^{\varepsilon_i}=d_j^{\varepsilon_j}$, for two distinct indices $i,j\in \{1,\ldots,\ell\}$. Let $m_1={m'}+d_i$ and let $\wp_1$ be the signed partition of $m-m_1$ of length $\ell-1$ obtained by removing $d_i^{\varepsilon_i}$ from $\wp$. As $\lcm(q^{d_i}-\varepsilon_i,q^{d_j}-\varepsilon_j)=q^{d_j}-\varepsilon_j$, we get $L_{m,q}({m'},\wp)\leq L_{m,q}(m_1,\wp_1)$. Now the claim follows by iterating this procedure.~$_\blacksquare$

\smallskip

\noindent\textsc{Claim~$2$. }Either $q>2$ and ${m'}=0$, or $q=2$ and ${m'}\leq 3$.

\smallskip

\noindent Applying both inequalities of Lemma~\ref{estimate} (first the lower bound to $M_m(q)$ and then the upper bound to $M_{m-m'}(q)$),~\eqref{eq00} and the induction on $m$, we have

\begin{eqnarray*}
q^m&<&M_{m}(q)\leq M=L_{m,q}({m'},\wp)=2^{\lceil \log_2(2{m'})\rceil}L_{m-{m'},q}(0,\wp)\\
&\leq&2^{\lceil\log_2(2{m'})\rceil}M_{m-{m'}}(q)\leq 2^{\lceil\log_2(2{m'})\rceil}\frac{q^{m-{m'}+1}-1}{q-1}
<2^{\lceil\log_2(2{m'})\rceil}\frac{q^{m-{m'}+1}}{q-1}
\end{eqnarray*}
and hence $q^{{m'}-1}(q-1)<2^{\lceil\log_2(2{m'})\rceil}$. An immediate computation gives ${m'}=0$ if $q>2$, and ${m'}\leq 3$ if $q=2$.~$_\blacksquare$

\smallskip

For $q>2$ and $m'=0$, and for $q=2$ and $m'\in \{0,1,2,3\}$, we see that $2^{\lceil\log_2(2m')\rceil}=2^{m'}$. Therefore, in view of Claim~$2$, we will replace $2^{\lceil\log_2(2m')\rceil}$ by $2^{m'}$ in the formula for $L_{m,q}(m',\wp)$.

\smallskip

\noindent\textsc{Claim~$3$. }For every two distinct $i,j\in \{1,\ldots,\ell\}$, we have $(q^{d_i}-\varepsilon_i,q^{d_j}-\varepsilon_j)=1$. 

\smallskip

\noindent We argue by contradiction and we assume that there exist $i,j\in \{1,\ldots,\ell\}$ with $s=(q^{d_i}-\varepsilon_i,q^{d_j}-\varepsilon_j)>1$. Observe that since $q$ is even, we have $s\geq 3$. 

Let $I=\{i\in \{1,\ldots,\ell\}\mid \varepsilon_i=-1\}$ and observe that, by Claim~$1$, the elements $(d_i)_{i\in I}$ are distinct. Hence we obtain

\begin{eqnarray*}\label{eq1}\nonumber
M&=&
L_{m,q}({m'},\wp)\leq \frac{2^{m'}}{s}\prod_{i\in  I}(q^{d_i}+1)\prod_{i\notin I}(q^{d_i}-1)\\
&\leq&
\frac{2^{m'}}{s} 
\prod_{i\in I}q^{d_i}\prod_{i\in  I}\left(1+\frac{1}{q^{d_i}}\right)\prod_{i\notin
  I}q^{d_i}=\frac{2^{m'}}{s}q^{m-m'}\prod_{i\in  I}\left(1+\frac{1}{q^{d_i}}\right)\\
&<&
\frac{2^{m'}}{s}q^{m-m'}\prod_{k=1}^\infty\left(1+\frac{1}{q^k}\right).\nonumber 
\end{eqnarray*}
Since $\log(1+x)\leq x$ for $x\geq 0$, we have
\begin{eqnarray*}
\log\left(\prod_{k=1}^\infty\left(1+\frac{1}{q^k}\right)\right)
&=&\sum_{k=1}^\infty\log\left(1+\frac{1}{q^k}\right)\leq\sum_{k=1}^\infty\frac{1}{q^k}=\frac{1}{q-1}. 
\end{eqnarray*}
Thus $M< (2^{m'}q^{m-m'}/s)\exp(1/(q-1))$. Moreover,  as $\exp(y)\leq 1+2y$ (which is valid for $0\leq y\leq 1$), we get $\exp(1/(q-1))\leq 1+2/(q-1)$ and hence
$$M<2^{m'}q^{m-m'}\left(\frac{1}{s}+\frac{2}{s(q-1)}\right).$$
By~\eqref{eq00} and by Lemma~\ref{estimate} we have $M>q^m$, and hence we  get
$$q^{{m'}}<2^{m'}\left(\frac{1}{s}+\frac{2}{s(q-1)}\right).$$
Now a computation (using $s\geq 3$) shows that this inequality is never satisfied.~$_\blacksquare$
 
\smallskip

Claim~$3$ shows that $M$ is simply the product $2^{m'}\prod_{i=1}^\ell(q^{d_i}-\varepsilon_i)$. Write $$I_-=\{d_i\mid i\in \{1,\ldots,\ell\}, \varepsilon_i=-1\}\quad \textrm{and}\quad I_+=\{d_i\mid i\in \{1,\ldots,\ell\},\varepsilon_i=1\}.$$ Lemma~\ref{basic}~(iii) yields $(d_i)_2\neq (d_j)_2$, for every two distinct elements $d_i,d_j\in I_-$. We use this remark frequently in the rest of the proof.

\smallskip

\noindent\textsc{Claim~$4$. }By replacing $\wp$ (if necessary), we may assume that $d$ is  a power of $2$ for every $d\in I_{-}$.

\smallskip

\noindent From Claim~$3$ and Lemma~\ref{basic}, the $2$-powers $((d)_2)_{d\in I_{-}}$ are pair-wise distinct and $(x)_2 \leq(y)_2$ for each $x\in I_{+}$ and $y\in I_{-}$. Let $\wp'$ be the signed partition of $m-{m'}$ obtained from $\wp$ by replacing the elements $(d^{-1})_{d\in I_{-}}$ with the summands in the $2$-adic expansion of $(\sum_{d\in I_{-}}d)$ and assigning sign $-1$ to each of these parts. Lemma~\ref{basic} and a moment's thought give that $L(m',\wp')=2^{m'}\prod_{d'^{\varepsilon'}\in\wp'}(q^{d'}-\varepsilon_{d'})$. Moreover, Lemma~\ref{babylonians} gives $M=L({m'},\wp)\leq L({m'},\wp')$. Hence we may replace $\wp$  with $\wp'$.~$_\blacksquare$

\smallskip

\noindent\textsc{Claim~$5$. }If $I_+=\emptyset$ and $|I_-|\geq 2$, then $M=M_m(q)$.

\smallskip

\noindent  Write $I_-=\{2^{x_1},\ldots,2^{x_t}\}$ with $x_1<\cdots<x_t$. As $I_+=\emptyset$, we get that $m-m'=2^{x_1}+\cdots+2^{x_t}$ is the $2$-adic expansion of $m-m'$. Suppose that $m-m'$ is even, that is, $x_1>0$. Note that $(m-m')\geq 6$ because $t=|I_-|\geq 2$. By~\eqref{eq0}, we have 
\begin{eqnarray*}
M&=&q^{m'}(q^{2^{x_1}}+1)\cdots (q^{2^{x_t}}+1)=q^m\prod_{i=1}^t\left(1+\frac{1}{q^{2^{x_i}}}\right)\leq q^m\prod_{j=0}^{x_t-x_1}\left(1+\frac{1}{(q^{2^{x_1}})^{2^j}}\right)\\
&=&q^mq^{2^{x_1}-2^{x_t+1}}\frac{q^{2^{x_t+1}}-1}{q^{2^{x_1}}-1}< q^mq^{2^{x_1}-2^{x_t+1}}\frac{q^{2^{x_t+1}}}{q^{2^{x_1}}-1}=q^m\frac{q^{2^{x_1}}}{q^{2^{x_1}}-1}.
\end{eqnarray*}

With $q$ being fixed, the function $x\mapsto q^{2^x}/(q^{2^{x}}-1)$ is decreasing with $x$. As $x_1>0$, we deduce that $M<q^m\cdot q^2/(q^2-1)$. Now, using the second statement in Lemma~\ref{estimate}, we get $M<M_{m}(q)$, which contradicts~\eqref{eq00}. Thus $m-m'$ is odd.

If $q>2$, then $m'=0$ from Claim~$2$, and hence $M=\prod_{i=1}^t(q^{2^{x_i}}+1)$. Observe now that $2^{x_1}+\cdots +2^{x_t}$ is the $2$-adic expansion of $m$. Since $m$ is odd, we see from Definition~\ref{def} (case $m$ odd and $q>2$) that $M$ equals $M_m(q)$. 

Suppose that $q=2$. We study separately three cases:
\begin{description}
\item[(i)]$m-m'=2^\ell-1$, for some $\ell\geq 1$;
\item[(ii)]$m-m'=2^\ell+2^{\ell-1}-1$, for some $\ell\geq 1$;
\item[(iii)]$m-m'\notin\{ 2^{\ell}-1,2^{\ell}+2^{\ell-1}-1\}$, for every $\ell\geq 1$.
\end{description}
Here we use~\eqref{eqdef} in Definition~\ref{def} and  we refer to each of its six lines as~$(1.1),\ldots,(1.6)$,  respectively.

Assume that $m-m'=2^\ell-1$, for some $\ell\geq 1$. Then $I_-=\{1,\ldots,2^{\ell-1}\}$ and $M=q^{m'}(q+1)\cdots (q^{2^{\ell-1}}+1)=q^{m'}(q^{2^\ell}-1)$. As $m'\leq 3$, by Lemma~\ref{comp} we may assume $2^{\ell}-1=m-m'>5$, that is, $\ell\geq 3$. From this it follows that $\ell$ is the largest positive integer with $2^\ell-1\leq m$. Now, we see from Definition~\ref{def}~$(1.1)$ that $M=M_m(q)$. 

Assume that $m-m'= 2^\ell+2^{\ell-1}-1$, for some $\ell\geq 1$.  As $m'\leq 3$, by Lemma~\ref{comp} we may assume $2^{\ell}+2^{\ell-1}-1=m-m'>5$, that is, $\ell\geq 3$. From this it follows that $\ell$ is the largest positive integer with $2^\ell-1\leq m$.

As the $2$-adic expansion of $m-m'=2^\ell+2^{\ell-1}-1$ is $1+2+\cdots +2^{\ell-2}+2^{\ell}$, we have $I_-=\{1,2,\ldots,2^{\ell-2},2^{\ell}\}$ and $M=q^{m'}((q+1)\cdots (q^{2^{\ell-2}}+1))(q^{2^\ell}+1)=q^{m'}(q^{2^\ell}+1)(q^{2^{\ell-1}}-1)$. If $m'=0$, then $M=M_m(q)$ by Definition~\ref{def}~$(1.4)$. Assume that $m'=1$. Then $m=2^{\ell}+2^{\ell-1}$ and hence $M_m(q)=(q^{2^{\ell-1}+1}-1)(q^{2^{\ell}}-1)$ by Definition~\ref{def}~$(1.5)$. Now a quick computation (using $\ell\geq 3$) shows that $M<M_m(q)$, which contradicts~\eqref{eq00}. Assume that $m'=2$. Then $m=2^{\ell}+2^{\ell-1}+1$ and hence $M_m(q)=q(q^{2^{\ell-1}+1}-1)(q^{2^{\ell}}-1)$ by Definition~\ref{def}~$(1.6)$. Another quick computation (using $\ell\geq 3$) shows that $M<M_m(q)$, which contradicts~\eqref{eq00}.  Assume that $m'=3$. Then $m=2^{\ell}+2^{\ell-1}+2$ and $M_m(q)=(q^{2^{\ell-1}+3}-1)(q^{2^{\ell}}-1)$  by Definition~\ref{def}~$(1.5)$. Another computation (using $\ell\geq 3$) shows that $M<M_m(q)$, which contradicts~\eqref{eq00} again. 

It remains to consider the case that $q=2$ and $m-m'\notin\{ 2^\ell-1,2^\ell+2^{\ell-1}-1\}$, for every $\ell\geq 1$. Observe that this means that, there exists $x\in \{1,\ldots, x_t-2\}$ with $2^x\notin I_-$, that is, $2^x$ is not a summand of the $2$-adic expansion $2^{x_1}+\cdots +2^{x_t}$ of $m-m'$. Suppose that $m'<3$ and  let $\wp'$ be the signed partition obtained from $\wp$ by adding $(2^x)^{-1}$ and $(2^{x_t}-2^x-1)^1$ and by removing $(2^{x_t})^{-1}$. Observe that $\wp'$ is a signed partition of $m-m'-1$. Now, it is easy to verify that $q^{2^{x_t}}+1<q(q^{2^{x}}+1)(q^{2^{x_t}-2^x-1}-1)$. From this, using $0<x\leq x_t-2$, $m'<3$ and Lemma~\ref{basic}, with a computation we find that $M=L_{m,q}(m',\wp)<L_{m,q}(m'+1,\wp')$, contradicting the maximality to $M$. Suppose then that $m'=3$.  Let $\wp'$ be the signed partition obtained from $\wp$ by adding $(2^x)^{-1}$ and $(2^{x_t}-2^x+1)^1$ and by removing $(2^{x_t})^{-1}$. Observe that $\wp'$ is a signed partition of $m-m'+1$. Now, it is easy to verify that $q(q^{2^{x_t}}+1)<(q^{2^{x}}+1)(q^{2^{x_t}-2^x+1}-1)$. From this, using $0<x\leq x_t-2$, $m'=3$ and Lemma~\ref{basic}, with a computation we find that $M=L_{m,q}(m',\wp)<L_{m,q}(m'-1,\wp')$, contradicting again the maximality to $M$.~$_\blacksquare$

\smallskip

\noindent\textsc{Claim~$6$. }If $|I_-|=1$ and $I_+=\emptyset$, then $M=M_m(q)$.

\smallskip

\noindent As $I_+=\emptyset$, we have $I_{-}=\{m-m'\}$ and $M=q^{m'}(q^{m-m'}+1)$. Suppose that $q>2$. Then $m'=0$ by Claim~$2$, and hence $m$ is a power of $2$ by Claim~$4$. If $m\in \{1,2\}$, then $M=q^m+1=M_{m}(q)$. If $m>2$, then  Definition~\ref{def} (case $m\geq 4$ even and $q>2$) gives $M_{m}(q)=(q^{m/2+1}-1)(q^{m/2}-1)/(q-1)$. Now a computation (using $m\geq 4$) shows that $(q^{m/2+1}-1)(q^{m/2}-1)/(q-1)>q^m+1=M$, however this contradicts~\eqref{eq00}.

Suppose that $q=2$. By Claim~$4$, we see that $m-m'$ is a power of $2$, say $m-m'=2^\ell$. As $m'\leq 3$, by Lemma~\ref{comp} we may assume that $2^\ell=m-m'>5$, that is, $\ell\geq 3$. From this it follows that $2^\ell$ is the largest power of $2$ with $2^\ell-1\leq m$. If  $m'\leq 2$, then  $M_m(q)=q^{m'+1}(q^{m-m'}-1)$ by Definition~\ref{def}~$(1.1)$. Now $M_m(q)>M=q^{m'}(q^{m-m'}+1)$, contradicting~\eqref{eq00}. If $m'=3$, then $m=2^\ell+3=(2^\ell-1)+4$ and hence $M_m(q)=q(q^{2^{\ell-1}+3}-1)(q^{2^{\ell-1}}-1)$ if $\ell\geq 4$ (by Definition~\ref{def}~$(1.3)$) and $M_m(q)=(q^8+1)(q^4-1)$ if $\ell=3$ (by Definition~\ref{def}~$(1.4)$). In both cases a computation shows that $M_m(q)>M$, contradicting~\eqref{eq00}.~$_\blacksquare$

\smallskip

In view of Claims~$5$ and~$6$, we may assume $I_+\neq\emptyset$. In spirit, the rest of the proof is similar to the proof of Claims~$5$ and~$6$. The main major difference is that it requires (unfortunately) more subcases and slightly more detailed computations. 

\smallskip

\noindent\textsc{Claim~$7$. }We have $|I_+|=1$. Moreover, for $q=2$, either $m'=0$, or $m'=1$ and the element of $I_+$ is odd.

\smallskip

\noindent  Suppose that $q>2$. If $d$ and $d'$ are two distinct elements of $I_{+}$, then $(q^{d}-1,q^{d'}-1)$ is divisible by $q-1>1$, which contradicts Claim~$3$.  Thus $|I_+|=1$.

Suppose that $q=2$ and write $d=\sum_{x\in I_+}x$. 
Assume that $d$ is odd and that $m'\geq 2$. Let $\wp'$ be the signed partition of $m-m'+2$ obtained from $\wp$ by removing the parts $(x^1)_{x\in I_+}$ and by adding $(d+2)^1$. Observe that $q^2\prod_{x\in I_+}(q^x-1)<q^{d+2}-1$. Now, using Lemma~\ref{basic} we get $L_{m,q}(m'-2,\wp')=q^{m'-2}(q^{d+2}-1)\prod_{y\in I_-}(q^y+1)$, from which it follows that $M=L_{m,q}(m',\wp)<L_{m,q}(m'-2,\wp')$. However, this contradicts the maximality of $M$. 

Assume that $d$ is even and that $m'\geq 1$. Let $\wp'$ be the signed partition of $m-m'+1$ obtained from $\wp$ by removing the parts $(x^1)_{x\in I_+}$ and by adding $(d+1)^1$. Observe that $q\prod_{x\in I_+}(q^x-1)<q^{d+1}-1$. Now, using Lemma~\ref{basic} we get $L_{m,q}(m'-1,\wp')=q^{m'-1}(q^{d+1}-1)\prod_{y\in I_-}(q^y+1)$, from which it follows that $M=L_{m,q}(m',\wp)<L_{m,q}(m'-1,\wp')$. However, this contradicts the maximality of $M$. 

Summing up, we have shown that either $d$ is odd and $m'\in \{0,1\}$, or $d$ is even and $m'=0$. In particular, to conclude the proof of this claim it suffices to show that $I_+=\{d\}$. We argue by contradiction and we suppose that $|I_+|\geq 2$. 

Assume that $d$ is odd. Let $\wp'$ be the signed partition of $m-m'$ obtained from $\wp$ by removing the parts $(x^1)_{x\in I_+}$ and by adding $(d)^1$. Observe that $\prod_{x\in I_+}(q^x-1)<q^{d}-1$. Now, using Lemma~\ref{basic} we get $L_{m,q}(m',\wp')=q^{m'}(q^{d}-1)\prod_{y\in I_-}(q^y+1)$, from which it follows that $M=L_{m,q}(m',\wp)<L_{m,q}(m',\wp')$. However, this contradicts the maximality of $M$. 

Assume that $d$ is even. Recall that $m'=0$.  Let $\wp'$ be the signed partition of $m-1$ obtained from $\wp$ by removing the parts $(x^1)_{x\in I_+}$ and by adding $(d-1)^1$. Observe that $\prod_{x\in I_+}(q^x-1)<q(q^{d-1}-1)$ because $|I_+|\geq 2$. Now, using Lemma~\ref{basic} we get $L_{m,q}(m',\wp')=q(q^{d-1}-1)\prod_{y\in I_-}(q^y+1)$, from which it follows that $M=L_{m,q}(m',\wp)<L_{m,q}(1,\wp')$. However, this contradicts again the maximality of $M$.~$_\blacksquare$ 

\smallskip

We denote by $d_+$ the element of $I_+$, that is, $I_+=\{d_+\}$. For  $q=2$, we have $q-1=1$ and 
$ q^2-1=q+1$, and hence (by eventually replacing $\wp$ with the signed partition obtained from $\wp$ by removing $2^1$ and by adding $1^{-1}$) we may assume that $d_+> 2$. Observe that in view of Claims~$2$ and~$7$ (at this stage of the proof) we have $m'=0$ if $q>2$, and $m'\leq 1$ if $q=2$.

\smallskip

\noindent\textsc{Claim~$8$. }$I_-\neq\emptyset$.

\smallskip

\noindent If $I_{-}=\emptyset$, then $d_+=m-m'$ and $M=q^{m'}(q^{m-m'}-1)<q^{m}< M_m(q)$ by Lemma~\ref{estimate}, which contradicts~\eqref{eq00}. Thus $I_{-}\neq \emptyset$.~$_\blacksquare$ 

\smallskip

Write $I_-=\{2^{x_1},\ldots,2^{x_t}\}$, with $x_1<\cdots<x_t$. Observe that $(d_+)_2\leq 2^{x_1}$ by Lemma~\ref{basic}. 

\smallskip

\noindent\textsc{Claim~$9$. }$d_+$ is odd.

\smallskip

\noindent We argue by contradiction and we suppose that $d_+$ is even. In particular, $x_1>0$ because $2\leq (d_+)_2\leq (2^{x_1})_2$ by Claim~$3$ and Lemma~\ref{basic}~(ii). Assume that $d_+>2$. Let $\wp'$ be the signed partition of $m-m'$ obtained from $\wp$ by removing $d_+^1$ and by adding  $1^{-1}$ and $(d_+-1)^{1}$. An application of Lemma~\ref{basic} gives $L_{m,q}(m',\wp')=q^{m'}(q+1)(q^{d_+-1}-1)\prod_{y\in I_-}(q^{y}+1)$. Moreover, as $1<d_+-1$, we get $q^{d_+}-1<(q+1)(q^{d_+-1}-1)$ and $M=L_{m,q}(m',\wp)<L_{m,q}(m',\wp')$, which contradicts the maximality of $M$. 

Assume that $d_+=2$. So $q>2$, and $m'=0$ by Claim~$2$. Let $s$ be the largest non-negative integer with $x_i=i$, for every $i\in \{1,\ldots,s\}$. (For instance, $s=0$ when $x_1>1$, and $s=1$ when $x_1=1$ and either $x_2>2$ or $t=1$.) Let $\wp'$ be the signed partition of $m$ obtained from $\wp$ by removing $2^1$ and $((2^i)^{-1})_{i\in \{1,\ldots,s\}}$ and by adding  $(2^{s+1})^{-1}$. Observe that this is well-defined because $2+(2^1+2^2+\cdots +2^{s})=2^{s+1}$. Moreover, by the maximality of $s$, we get either $s=t$ or $x_{s+1}>s+1$. In both cases, $2^{s+1}\notin I_-$, and hence an application of Lemma~\ref{basic} gives $L_{m,q}(0,\wp')=(q^{2^{s+1}}+1)\prod_{i=s+1}^t(q^{2^{x_i}}+1)$. Furthermore, as $$(q^{2}-1)(q^2+1)\cdots (q^{2^s}+1)=q^{2^{s+1}}-1<q^{2^{s+1}}+1,$$ we
have  $M=L_{m,q}(0,\wp)<L_{m,q}(0,\wp')$, which contradicts again the maximality of $M$. This final contradiction shows that $d_+$ must be odd.~$_\blacksquare$

\smallskip

\noindent\textsc{Claim~$10$. }$1\in I_-$, that is, $x_1=0$.

\smallskip

\noindent Suppose that $1\notin I_-$. In particular, no element in $I_-$ is odd. Let $\wp'$ be the signed partition obtained from $\wp$ by replacing $d_+^{1}$ with $d_+^{-1}$. Since $d_+$ is odd, Lemma~\ref{basic}~(iii) gives $L_{m,q}(m',\wp')=q^{m'}(q^{d_+}+1)\prod_{y\in I_-}(q^y+1)$. Moreover, as $q^{d_+}-1<q^{d_+}+1$, we obtain $L_{m,q}(m',\wp')>L_{m,q}(m',\wp)=M$, which contradicts the maximality of $M$. Thus $1\in I_-$ and $x_1=0$.~$_\blacksquare$  

\smallskip

\noindent\textsc{Claim~$11$. }If $|I_-|\geq 2$, then $M=M_m(q)$.

\smallskip

\noindent  Suppose that $d_+\leq 2^{x_t}$ and write $d_+'=d_++2^{x_t}$. Observe that $(d_+')_2=(d_+)_2=1$ because $d_+$ is odd and $x_t>x_1=0$. Let $\wp'$ be the signed partition obtained from $\wp$ by removing $d_+^1$ and $(2^{x_t})^{-1}$ and by adding $(d_+')^1$. As $(q^{d_+}-1)(q^{2^{x_t}}+1)<q^{d_+'}-1$, the usual application of Lemma~\ref{basic} gives $$L_{m,q}(m',\wp')=q^{m'}(q^{d_+'}-1)\prod_{\substack{y\in I_-\\y\neq 2^{x_t}}}(q^y+1)>L_{m,q}(m',\wp)=M,$$ which is a contradiction. Thus 
\begin{equation}\label{eqnew1}
2^{x_t}\leq d_+.
\end{equation}

Let $x\geq 0$ with $2^{x+1}\leq d_+$. Suppose that $2^x\notin I_-$.  Let $\wp'$ be the signed partition obtained from $\wp$ by removing $d_+^1$  and by adding $(d_+-2^x)^1$ and $(2^x)^{-1}$. As $q^{d_+}-1<(q^{d_+-2^x}-1)(q^{2^x}+1)$, Lemma~\ref{basic} gives $L_{m,q}(m',\wp')=q^{m'}(q^{d_+-2^x}-1)(q^{2^x}+1)\prod_{y\in I_-}(q^y+1)>L_{m,q}(m',\wp)=M$, which is a contradiction. Thus $2^{x}\in I_-$. Therefore
\begin{equation}\label{eqnew2}
2^x\in I_-\qquad\textrm{ for every }x\geq 0\textrm{ with }2^{x+1}\leq d_+.
\end{equation}

Let $\ell$ be the largest integer with $2^{\ell}\leq d_+$. From~\eqref{eqnew2}, we have $1,2,\ldots,2^{\ell-1}\in I_-$. Now, combining~\eqref{eqnew1} and~\eqref{eqnew2}, we get that  either
\begin{description}
\item[(i)]$x_t={\ell-1}$ and $I_-=\{1,2,\ldots,2^{\ell-1}\}$, or
\item[(ii)]$x_t=\ell$ and $I_-=\{1,2,\ldots,2^\ell\}$.
\end{description}

To discuss these two possibilities we subdivide the proof depending on whether $q>2$ or $q=2$. Suppose first that $q>2$. In particular, $m'=0$. Assume~(i), that is, $x_t={\ell-1}$ and $I_-=\{1,2,\ldots,2^{\ell-1}\}$. Thus $M=(q^{d_+}-1)\prod_{i=0}^{\ell-1}(q^{2^i}+1)=(q^{d_+}-1)(q^{2^\ell}-1)/(q-1)$. Moreover, $m=d_++1+2+\cdots +2^{\ell-1}=d_++2^{\ell}-1$. Since $d_+$ is odd, we see that $m$ is even. As $\ell$ is the largest integer with $2^\ell\leq d_+$, from an easy computation, we see that  $\ell$ is also the largest integer with  $2^{\ell-1}+2^{\ell}\leq m$. Now, Definition~\ref{def} (case $m\geq 4$ even and $q>2$) gives $M=M_m(q)$. Assume~(ii), that is, $x_t=\ell$ and $I_-=\{1,2,\ldots,2^\ell\}$. Thus $M=(q^{d_+}-1)\prod_{i=0}^{\ell}(q^{2^i}+1)=(q^{d_+}-1)(q^{2^\ell+1}-1)/(q-1)$. Moreover, $m=d_++1+2+\cdots +2^{\ell}=d_++2^{\ell+1}-1$. Since $d_+$ is odd, $m$ is even. As $\ell$ is the largest integer with $2^\ell\leq d_+$ and as $d_+$ is odd, we deduce that $\ell+1$ is the largest integer with  $2^{\ell}+2^{\ell+1}\leq m$. So, Definition~\ref{def} (case $m\geq 4$ even and $q>2$) gives again $M=M_m(q)$.

Suppose that $q=2$. Assume~(i), that is, $x_t=\ell-1$ and $I_-=\{1,2,\ldots,2^{\ell-1}\}$. Thus $M=q^{m'}(q^{d_+}-1)(q^{2^\ell}-1)$ and $m=m'+d_++2^\ell-1$. As $\ell$ is the largest integer with $d_+\geq 2^\ell$, with an easy computation we see that $\ell+1$ is the largest integer with $2^{\ell+1}-1\leq m$. Write $m_0=m-(2^{\ell+1}-1)=m'+d_+-2^\ell$. Observe that $m_0\leq 2^\ell$ because $d_+<2^{\ell+1}$ and $m'\leq 1$ by Claims~$2$ and~$7$. Now, Definition~\ref{def}~(1.1),~(1.2),~(1.3) and~(1.4) gives
\begin{equation}\label{long}
M_m(q)=\left\{
\begin{array}{lcl}
q^{m_0}(q^{2^{\ell+1}}-1)&&\textrm{if }m_0\leq 3,\\
(q^{m_0+2^\ell}-1)(q^{2^{\ell}}-1)&&\textrm{if }3<m_0<2^\ell \textrm{ and }m_0 \textrm{ is odd},\\
q(q^{m_0+2^\ell-1}-1)(q^{2^{\ell}}-1)&&\textrm{if }3<m_0<2^\ell \textrm{ and }m_0 \textrm{ is even},\\
(q^{2^\ell+1}-1)(q^{2^{\ell}}-1)&&\textrm{if } 3<m_0=2^\ell.
\end{array}
\right.
\end{equation}
If $m_0\leq 3$ or if $m_0=2^\ell$, then a direct computation shows that $M<M_m(q)$, contradicting~\eqref{eq00}. As $d_+$ is odd, $m_0$ is odd if and only if $m'=0$, and $m_0$ is even if and only if $m'=1$. Thus from~\eqref{long} we get that 
\[
M_m(q)=\left\{
\begin{array}{lcl}
(q^{d_+}-1)(q^{2^\ell}-1)&&\textrm{if }m'=0,\\
q(q^{d_+}-1)(q^{2^\ell}-1)&&\textrm{if }m'=1,
\end{array}
\right.
\]
which is exactly $M$.

Assume~(ii), that is, $x_t=\ell$ and $I_-=\{1,2,\ldots,2^\ell\}$. Thus $M=q^{m'}(q^{d_+}-1)(q^{2^{\ell+1}}-1)$ and $m=m'+d_++2^{\ell+1}-1$. As $\ell$ is the largest integer with $d_+\geq 2^\ell$, with an easy computation we see that either $\ell+1$ is the largest integer with $2^{\ell+1}-1\leq m$, or $(d_+,m')=(2^{\ell+1}-1,1)$. In the second case we have $m=2^{\ell+2}-1$ and $M_m(q)=q^{2^{\ell+2}}-1$ by Definition~\ref{def}~$(1.1)$. Now a computation shows that $M<M_m(q)$, contradicting~\eqref{eq00}. In the first case, write $m_0=m-(2^{\ell+1}-1)=m'+d_+$ and observe that $m_0>2^\ell$ because $d_+\geq 2^\ell$ and $d_+$ is odd. Thus Definition~\ref{def}~$(1.5)$ and~$(1.6)$ gives
\[
M_m(q)=\left\{
\begin{array}{lcl}
(q^{m_0}-1)(q^{2^\ell+1}-1)&&\textrm{if }m_0 \textrm{ is odd},\\
q(q^{m_0-1}-1)(q^{2^\ell+1}-1)&&\textrm{if }m_0 \textrm{ is even}.
\end{array}
\right.
\]
Finally, as $m_0$ is odd if and only if $m'=0$, and $m_0$ is even if and only if $m'=1$, we get $M=M_m(q)$.~$_\blacksquare$

\smallskip

In view of Claims~$8$ and~$11$ there is only one more case to consider: the case $|I_-|=1$. By Claim~$10$, we have $I_-=\{1\}$, and hence $d_+=m-m'-1$ and $M=q^{m'}(q^{m-m'-1}-1)(q+1)$. Recall that $d_+=m-m'-1$ is odd by Claim~$9$. 

Suppose that $m-m'>5$. Let $\wp'$ be the signed partition $(1^{-1},2^{-1},(m-m'-3)^{1})$ of $m-m'$. As $m-m'-3$ is odd, Lemma~\ref{basic} gives $L_{m,q}(m',\wp')=q^{m'}(q^{m-m'-3}-1)(q+1)(q^2+1)$. Now a direct computation using $m-m'>5$ gives $L_{m,q}(m',\wp')>L_{m,q}(m',\wp)=M$, contradicting the maximality of $M$. In particular, we may assume that $m-m'\leq 5$. Moreover, by Lemma~\ref{comp} we may also assume $q>2$. In this case, as $m'=0$, we have $m\leq 5$ and $M=(q^{m-1}-1)(q+1)$. Since $d_+=m-1$ is odd, $m$ is even. Therefore to conclude it suffices to check the values of $M_2(q)$ and $M_4(q)$. Now, $M_2(q)=q^2+1>q^2-1=M$, contradicting~\eqref{eq00}, and $M_4(q)=(q^3-1)(q+1)=M$. The proof is now complete.
\end{proof}

\end{document}